\documentclass[a4paper, 12pt]{amsart}
\usepackage{amssymb, amsmath}

\theoremstyle{plain}

\newtheorem{thm}{Theorem}[section]
\newtheorem{lem}[thm]{Lemma}
\newtheorem{cor}[thm]{Corollary}

\newtheorem{Def}[thm]{Definition}

\newtheorem{rem}[thm]{Remark}
\newcommand{\Rem}{\begin{rem} \rm}
\newcommand{\bdfn}{\begin{Def} \rm}
\newcommand{\edfn}{\end{Def}}

\begin{document}
\title{inclusion of  $\Lambda BV^{(p)}$  spaces in the classes $H_{\omega}^{q}$ }
\author[M. Hormozi]{Mahdi Hormozi}
\address[Mahdi Hormozi ]
{ Department of Mathematical Sciences, Division of Mathematics\\ Chalmers University
of Technology and University of Gothenburg\\ Gothenburg 41296, Sweden}
\email[Mahdi Hormozi]{hormozi@chalmers.se}
\date{\today}
\begin{abstract}
A characterization of the inclusion of Waterman-Shiba classes into
classes of functions with given integral modulus of continuity is given. This corrects and extends an earlier result of a paper from 2005.
\end{abstract}

\maketitle

\section{Preliminaries}
Let $\:\Lambda=(\lambda_i)\:$ be a nondecreasing sequence of positive numbers such that
$\:\sum\frac1{\lambda_i}=+\infty\:$ and let $p$ be a number greater than or equal to 1. A function $\:f:[a,\,b]\to\mathbb R\:$ is said to be of bounded $p$-$\Lambda$-variation
on a not necessarily closed subinterval $\:P\subset[a,\,b]\:$ if
$$
 V(f)\ :=\ \sup\left(\sum_{i=1}^n\frac{|f(I_i)|^p}{\lambda_i}
\right)^\frac1p\ <\ +\infty,
$$
where the supremum is taken over all finite families $\:\{I_i\}_{i=1}^n\:$ of
nonoverlapping subintervals of $P$ and where $\:f(I_i):=f(\sup I_i)-f(\inf I_i)\:$
is the change of the function $f$ over the interval $I_i$. The symbol $\:\Lambda BV^{(p)}\:$ denotes
the linear space of all functions of bounded $p$-$\Lambda$-variation with domain $[0,\,1]$. The Waterman-Shiba class
$\:\Lambda BV^{(p)}\:$ was introduced in 1980 by M. Shiba in \cite{S}. When $p=1$, $\:\Lambda BV^{(p)}\:$ is the well-known Waterman class $\:\Lambda BV$. Some of the basic
properties of functions of class $\:\Lambda BV^{(p)}\:$ were discussed by R. G. Vyas
in \cite{V3} recently. More results concerned with the Waterman-Shiba classes and their
applications can be found in \cite{BT}, \cite{BTV},  \cite{HLP}, \cite{L}, \cite{SW1}, \cite{SW}, \cite{V1}
and \cite{V2}. $\Lambda BV^{(p)}\:$ equipped with the norm $\:\|f\|_{\Lambda,\,p}:=
|f(0)|+V(f)\:$ is a Banach space.

Functions in a Waterman-Shiba class $\Lambda BV^{(p)}$ are
regulated \cite[Thm. 2]{V2}, hence integrable, and thus it makes sense to consider their integral modulus of continuity
$$
\omega_q(\delta,\,f)\ :=\ \sup_{0\le \gamma \le\delta}\,\left( \int_0^{1-\gamma}|f(t+\gamma)-f(t)|^q\ \right)^\frac1q dt,
$$
for $\:0\le\delta\le1$. However, if $f$ is defined on $\:\mathbb R\:$ instead of on $[0,\,1]$
 and if $f$ is 1-periodic, it is convenient to modify the definition and put
 $$
 \omega_q(\delta,\,f)\ :=\ \sup_{0\le \gamma \le\delta}\,\left(\int_0^1|f(t+\gamma)-f(t)|^q\ \right)^\frac1q dt,
 $$
 since the difference between the two definitions is then nonessential in all applications
 of the concept. A function $\:\omega:\,[0,\,1]\to\mathbb R\:$ is said to be a modulus of continuity if it is
 nondecreasing, continuous and $\:\omega(0)=0$. If $\omega$ is a modulus of continuity,
 then $H_{\omega}^{q}$ denotes the class of functions $\:f\in L^q[0,\,1]\:$ for which $\:\omega_q(
 \delta,\,f)=\text{O}(\omega(\delta))\:$ as $\:\delta\to0+$.\\

 In \cite{HLP}, a necessary and sufficient condition for the inclusion $\Lambda BV^{(p)}$ in $H_{\omega}^{1}$, is given. Also, Wang \cite{Wg} by using an interesting method found a necessary and sufficient condition for the embedding $\: H_{\omega}^{q} \subset \Lambda BV$. Here, we give a necessary and sufficient condition for the inclusion of $\Lambda BV^{(p)}$ in $H_{\omega}^{q}$.

 \section{main result}
 In  \cite {G2}, it was claimed that the following is true.
 \begin{thm}
   \label{t1}
   For $q \in [1 , \infty)$, the inclusion $\Lambda BV \subset H_{q}^{\omega}$ holds if and only if
   \begin{equation}
    \label{e1}
   \limsup_{n\rightarrow\infty} \frac{1}{\omega (1/n) n^{\frac{1}{q}}}\ \max_{1\leq k\leq n} \ \frac{k^{\frac{1}{q}}}{(\sum_{i=1}^{k} \frac{1}{\lambda_i})}< +\infty.
   \end{equation}
   \end{thm}
The proof of sufficiency of the condition came up immediately from \cite{kup} and so, the main part of  \cite{G2} concerns the proof of necessity. The theorem itself is correct but, unfortunately, there is a major mistake regarding the existence of some subsequences, which ensures that the proof of \cite{G2} is incorrect. To understand this, take $\omega(x)=x^{1/p}$. Then we can choose $\gamma_k=2^k, \gamma_k'=n_k'=k$. If we take $s_k'=k$, then the condition for case
(a) in \cite{G2} is satisfied, since $m(x)\leq 2^x$ by definition. But for subsequences $r_k$ of $s'_k$, relation (3) in \cite{G2}, $\omega(\frac{1}{2^{r_k}})\cdot 2^{r_k/p} \leq 4^{-k}$, is not true.\\

Our main result provides a characterization of the embedding of a generalization of $\Lambda BV$, Waterman-Shiba classes, into classes of functions with given integral modulus of continuity. Thus, by considering $p=1$, the correctness of \cite[Thm. 1]{G2} can be verified.
   \begin{thm}
   \label{t2}
   For $p,q \in [1 , \infty)$, the inclusion $\Lambda BV^{(p)}\subset H_{q}^{\omega}$ holds if and only if
   \begin{equation}
    \label{e2}
   \limsup_{n\rightarrow\infty}  \left \{\frac{1}{\omega (1/n) n^{\frac{1}{q}}}\ \max_{1\leq k\leq n} \ \frac{k^{\frac{1}{q}}}{(\sum_{i=1}^{k} \frac{1}{\lambda_i})^{\frac{1}{p}}} \right \}< +\infty.
   \end{equation}
   \end{thm}
\begin{proof}
To observe that equation (2) is a \textbf{sufficiency} condition for the inclusion $\Lambda BV^{(p)}\subset H_{q}^{\omega}$, we prove  an inequality which gives us the sufficiency :
$$\omega (\frac{1}{n},f)_{q} \leq  ~V(f)\left \{\frac{1}{n}\ \max_{1\leq k\leq n}\ \frac{k}{(\sum_{i=1}^{k} 1/\lambda_{i})^{\frac{q}{p}}} \right\} ^{\frac{1}{q}}.$$
Kuprikov \cite{kup} obtained Lemma \ref{l1} and Corollary \ref{c1} where $q\geq1$.
 \begin{lem}
 \label{l1}
Let $q \geq 1 $ and suppose $F(x)=\sum_{i=1}^{n} x_{i} ^{q}$ takes its maximum value under the following conditions
 \begin{equation}
 \nonumber
 (\sum_{i=1}^{n}\frac{x_i}{\lambda_i}) \leq 1,
 \end {equation}
   \begin{equation}
   \nonumber
   x_1 \geq x_2 \geq x_3 \geq ... \geq x_n\geq 0,
 \end {equation}
  then, $x=(x_1,x_2,...,x_n)$ satisfy
  \begin{equation}
  \label{e4}
 x_1= x_2= ...= x_k=\frac{1}{\sum_{j=1}^{k} 1/\lambda_{j}}> x_{k+1}= x_{k+2}=...= x_n=0,
  \end {equation}
 for some $k$~$( 1\leq k\leq n)$.
  \end{lem}
  \begin{cor}
\label{c1}
The maximum value of $F(x)$, under the conditions of Lemma \ref{l1}, is
$\max_{1 \leq k \leq n}\ \frac{k}{(\sum_{j=1}^{k} \tfrac1{\lambda_{j}})^{q}}.$
\end{cor}
  \begin{lem}
  \label{l2}
Suppose $0 <q <1$ and the conditions of Lemma \ref{l1} hold, then
 $$
 \sum_{i=1}^{n}x^q_i  = \frac{n}{(\sum_{j=1}^{n} \tfrac1{\lambda_{j}})^{q}}= ~\max_{1 \leq k \leq n}\ \frac{k}{(\sum_{j=1}^{k} \tfrac1{\lambda_{j}})^{q}}.
 $$
\end{lem}
\begin{proof}H\"{o}lder inequality yields
$$
\sum_{i=1}^{n}x^q_i \leq n^{1-q}\sum_{i=1}^{n}x_i \leq \frac{n}{(\sum_{j=1}^{n} \tfrac1{\lambda_{i}})^{q}}.
$$
Thus $F(x)$ takes its maximum when $x_i=\frac{1}{\sum_{j=1}^{n} \tfrac1{\lambda_{j}}}$ for $1 \leq j\leq n$. 
\end{proof}
  
Now, we return to the proof of inequality:
\begin{align*}
  \omega_q(\tfrac1n,\,f)^{q}&=\ \ \sup_{0< h\le\frac1n}\,\int_0^1|f(x+h)-f(x)|^q\,dx\\[.1in]
    &=\ \sup_{0< h\le\frac1n}\, \int_0^\frac1n\,\sum_{k=1}^n |f(x+\tfrac{k-1}n+h)-f(x+\tfrac{k-1}n)|^q\,dx.
    \end{align*}
For $h\le\frac1n$ and fixed $x$, denote $x_k:=|f(x+\tfrac{k-1}n+h)-f(x+\tfrac{k-1}n)|^p$. We reorder $x_k$ such that \\
$$ x_1 \geq x_2 \geq ...\geq x_n \geq 0,\qquad \left(\sum_{k=1}^n  \frac{x_k}{\lambda _k} \right)^{\frac{1}{p}} \leq V(f).$$
Therefore, by replacing $q$ by $q/p$ in Lemma \ref{l1}, Lemma \ref{l2} and Corollary \ref{c1}, we get
\begin{align*}
  \omega_q(\tfrac1n,\,f)^{q}&=\  \sup_{0< h\le\frac1n}\, \int_0^\frac1n\, \sum_{k=1}^n x_k ^{\frac{q}{p}} \ dx\\[.1in]
  &\le\  \int_0^\frac1n\,V^q(f)\max_{1\leq k\leq n} \frac{k}{(\sum_{i=1}^{k} 1/\lambda_{i})^{\frac{q}{p}}} dx\\[.1in]
  &\overset{\text{}}{=}\ \frac{1}{n} V^q(f) \max_{1\leq k\leq n} \frac{k}{(\sum_{i=1}^{k} 1/\lambda_{i})^{\frac{q}{p}}}.
\end{align*}

  \textbf{ Necessity.} Suppose (\ref{e2}) doesn't hold, that is, there are sequences $n_k$ and $m_k$ such that
 \begin{equation}
\label{e5}
n_k \geq 2^{k+2},
\end{equation}
 \begin{equation}
\label{e7}
 m_k\leq n_k,
\end{equation}
\begin{equation}
\label{e6}
\omega(\tfrac1{n_k})(\frac{n_k}{m_k})^{\frac{1}{q}}\ (\sum_{i=1}^{m_k} \frac{1}{\lambda_i})^{\frac{1}{p}} < \frac{1}{2^{4k}},
\end{equation}
where
$$
\max_{1\leq \rho \leq n_k}\,\ \frac{\rho}{(\sum_{i=1}^{\rho} 1/\lambda_{i})^{\tfrac{1}{p}}} = \frac{m_k}{(\sum_{i=1}^{m_k} 1/\lambda_{i})^{\tfrac{1}{p}}}.
$$
Denote
\begin{equation}
\label{e100}
 \Phi_{k} :=  \frac{1}{\sum_{i=1}^{m_k} 1/\lambda_{i}}.
\end{equation}
Consider
$$ g_k(y):=\begin{cases}
 2^{-k}\Phi_{k}^{1/p} ~~~~~, ~~~~y\in [\tfrac1{2^k}+\frac{2j-2}{n_k},\tfrac1{2^k}+\frac{2j-1}{n_k}) ;~~~~~  1\leq j\leq N_k, \\[.1in]
0~~~~~\qquad\qquad \textmd{otherwise},\\
\end{cases}
$$
where
\begin{equation}
\label{e8}
s_k= \max \{j\in \mathbb{N} : 2j \leq \frac{n_k}{2^k}+1\},
\end{equation}
and
\begin{equation}
\label{e9}
N_k= \min \{m_k, s_k\}.
\end{equation}

Hence, applying the fact $2(s_k+1)\geq \frac{n_k}{2^k}+1$ and (\ref{e5}), we have
\begin{equation}
\label{e10}
 \frac{2s_k-1}{n_k} \geq 2^{-k-1}.
\end{equation}

The functions $g_k$ have disjoint support. Thus $g:=\sum_{k=1}^{\infty}g_k$ is a well-defined  function on $[0,\,1]$. Since
\begin{align*}
\|g\| &\leq \ \sum_{k=1}^{\infty}\|g_k\|\\[.1in]
&= \sum_{k=1}^{\infty}  \left(\sum_{j=1}^{2N_k} \frac{|2^{-k}\Phi^{1/p}_{k}|^p}{\lambda_j}\right)^{1/p} \\[.1in]
&\leq \sum_{k=1}^{\infty} 2^{-k} \left(2 \sum_{j=1}^{N_k} \frac{|\Phi_{k}|}{\lambda_j}\right)^{1/p} \\[.1in]
&\overset{\text{(\ref{e9})}}{\leq}\ \sum_{k=1}^{\infty} 2^{-k} \left(2 \sum_{j=1}^{m_k} \frac{|\Phi_{k}|}{\lambda_j}\right)^{1/p} \\[.1in]
&\overset{\text{(\ref{e100})}}{=} \ \sum_{k=1}^{\infty} 2^{-k} \left(2\  \frac{ \sum_{j=1}^{m_k} \frac{1}{\lambda_j}}{ \sum_{j=1}^{m_k} \frac{1}{\lambda_j}}\right)^{1/p} \\[.1in]
  &<\ +\infty,
\end{align*}
we observe that $g \in \Lambda BV^{(p)}$.\\

If $N_k=m_k$, then
\begin{equation}
\label{e11}
\frac{   (2N_k-1) }{n_k} \cdot |\Phi_{k}|^{\frac{q}{p}} \geq \frac{1}{(\tfrac{n_k}{m_k}) (\sum_{i=1}^{m_k} 1/\lambda_{i})^{\tfrac{q}{p}}}\\
\end{equation}
and if $N_k=s_k$ then
\begin{equation}
\label{e12}
\frac{   (2N_k-1) }{n_k} \cdot |\Phi_{k}|^{\frac{q}{p}} \overset{\text{(\ref{e7}),(\ref{e10})}}{\geq}\ \ 2^{-k-1} \cdot \frac{m_k}{n_k(\sum_{i=1}^{m_k} 1/\lambda_{i})^{\tfrac{q}{p}}}.
\end{equation}
Since $ |g(x+\frac1{n_k})-g(x)|= 2^{-k}\Phi_{k}^{1/p}$ for $\:x\in[\tfrac{1}{2^k},\,\tfrac1{2^k}+\frac{2N_k-1}{n_k}]$, we get
\begin{align*}
\omega_q(\tfrac{1}{n_k},\,g)^{q}&=\ \ \sup_{0< \gamma\le\tfrac{1}{n_k}}\,\int_0^1|g(x+\gamma)-g(x)|^q\,dx\\
&\geq \ \int_0^1|g(x+ \frac{1}{n_k})-g(x)|^q\,dx\\[.1in]
&\geq  \int_{\frac{1}{2^k}}^{\tfrac1{2^k}+\frac{2N_k-1}{n_k}}|g(x+\frac1{n_k})-g(x)|^q\,dx\\
&= \frac{2N_k-1}{n_k} \cdot  2^{-kq}|\Phi_{k}|^{\frac{q}{p}} \\
&\overset{\text{(\ref{e11}),(\ref{e12})}}{\geq}\  \frac1{2^{k+qk+1}} \cdot \frac{m_k}{n_k(\sum_{i=1}^{m_k} 1/\lambda_{i})^{\tfrac{q}{p}}},
\end{align*}
and finally
 \begin{align*}
\frac{\omega_q(\tfrac1{n_k},\,g)}{\omega(\tfrac1{n_k})}\ &\ge\ ( \frac1{2^{k+qk+1}})^{\tfrac1q} \cdot \frac{1}{\omega(\tfrac1{n_k})} \cdot  \left( \frac{1}{(\tfrac{n_k}{m_k})^{\tfrac{1}{q}} (\sum_{i=1}^{m_k} 1/\lambda_{i})^{\tfrac{1}{p}}} \right) \\
&\overset{\text{(\ref{e6})}}{\geq}\  2^k \xrightarrow[k\to\infty]{}
\ +\infty,
\end{align*}
which shows that $g\notin H_{\omega}^{q}$.
\end{proof}
Condition (\ref{e2}) simplifies when $p\geq q$. Thus we have the following Corollary.
\begin{cor}
   \label{t2}
   For $p,q \in [1 , \infty) (p\geq q)$, the inclusion $\Lambda BV^{(p)}\subset H_{q}^{\omega}$ holds if and only if
   \begin{equation}
    \label{e40}
   \limsup_{n\rightarrow\infty}  \left \{ \frac{1}{\omega (1/n) (\sum_{i=1}^{n} \frac{1}{\lambda_i})^{\frac{1}{p}}} \right \}< +\infty.
   \end{equation}
   \end{cor}
   $$
   $$
\section{ACKNOWLEDGMENTS}
The author is so grateful to Professor Hjalmar Rosengren for valuable comments, helpful discussions and for reviewing earlier drafts very carefully. The author also thanks Peter Hegarty for pointing out a  mistake in an earlier version.


\begin{thebibliography}{K}
 \bibitem{BT}
W. W. Breckner, T. Trif, \emph{On the singularities of certain families of
nonlinear mappings}, Pure Math. Appl. 6 (1995) 121--137
\bibitem{BTV}
W. W. Breckner, T. Trif, C. Varga, \emph{Some applications of the condensation of the
singularities of families of nonnegative functions}, Anal. Math. 25 (1999) 12--32
\bibitem{G2}
U. Goginava, \emph{ On the embedding of the Waterman class in the class $H_p^\omega$},
Ukrainian Math. J. 57 (2005) 1818--1824
\bibitem {HLP}
M. Hormozi, A. A. Ledari, F. Prus-Wisniowski, \emph{On p-$\Lambda$-bounded variation}, Bull. Iranian Math. Soc.  37(4) (2011)  29--43
\bibitem{kup}
Y. E. Kuprikov, \emph{Moduli of continuity of functions from Waterman classes}, Moscow Univ. Math. Bull. 52(5) (1997) 46--49
\bibitem{L}
L. Leindler, \emph{A note on embedding of classes $H^\omega$}, Anal. Math. 27 (2001) 71--76
\bibitem{SW1}
M. Schramm, D. Waterman, \emph{ On the magnitude of Fourier coefficients}, Proc. Amer. Math.
Soc. 85 (1982) 407--410
\bibitem{SW}
M. Schramm, D. Waterman, \emph{ Absolute convergence of Fourier series of functions of
$\Lambda BV^{(p)}$ and $\Phi\Lambda BV$}, Acta Math. Hungar. 40 (1982) 273--276
\bibitem{S}
M. Shiba, \emph{On the absolute convergence of Fourier series of functions of class
$\Lambda BV^{(p)}$}, Sci. Rep. Fukushima Univ. 30 (1980) 7--10
\bibitem{V1}
R. G. Vyas, \emph{ On the absolute convergence of small gaps Fourier series of fuctions
of $\Lambda BV^{(p)}$}, J. Inequal. Pure Appl. Math. 6(1) (2005) Article 23
\bibitem{V2}
R. G. Vyas, \emph{On the convolution of functions of generalized bounded variation},
Georgian Math. J. 13 (2006) 193--197
\bibitem{V3}
R. G. Vyas, \emph{Properties of functions of generalized bounded variation}, Mat. Vesnik
58 (2006) 91--96
\bibitem{Wg}
H. Wang, \emph{Embedding of Lipschitz classes into classes of functions of $\Lambda$-bounded variation}, J. Math. Anal. Appl.  354(2) (2009)  698--703
\bibitem{Wt1}
D. Waterman, \emph{On  convergence of Fourier series of functions
 of bounded generalized variation}, Studia Math. 44 (1972) 107--117
\bibitem{Wt2}
D. Waterman, \emph{On $\Lambda$-bounded variation}, Studia Math. 57 (1976) 33--45
\end{thebibliography}
\end{document}